\documentclass[11pt,reqno,a4paper]{amsart}

\usepackage{colonequals}

\usepackage{amssymb,amsmath,amsthm,newlfont,enumerate}
\usepackage{bbold}
\usepackage{a4wide}
\usepackage{enumitem}

\usepackage{tikz}
\usetikzlibrary{positioning}
\usepackage[%% pdftex,% might be luatex, just allow automatic default
colorlinks, hyperindex, plainpages=false, bookmarksopen, bookmarksnumbered, pdfusetitle,  linkcolor=blue, citecolor=blue ]{hyperref}

\newtheorem{theorem}{Theorem}%  meant for continuous numbers
\newtheorem{proposition}[theorem]{Proposition}% 
\newtheorem{theo}{Theorem}[section]

\newtheorem{lemma}{Lemma}[section]

\newcommand{\RR}{\mathbb{R}}

\title[Schauder--Orlicz-Type Estimates for Divergence-Form Elliptic Equations]{Schauder--Orlicz-Type Estimates for Divergence-Form Elliptic Equations with Lower-Order Terms}
\author{J. Bourabiaa, Y. Elmadani, A. Hanine}
\subjclass[2020]{35B45, 35J25, 46E30}
\keywords{ Schauder-type estimates, divergence-form elliptic equations, Orlicz spaces}
\date{}
\address{Laboratory of Mathematical Analysis and Applications,
	Mohammed V University in Rabat,
	B.P. 1014, Rabat, Morocco}
\email{jaouad.bourabiaa@gmail.com}
\email{elmadanima@gmail.com}
\email{abhanine@gmail.com}
\begin{document} 
	\begin{abstract}
Schauder–Orlicz-type estimates are derived for weak solutions to second-order linear elliptic equations in divergence form with lower-order terms. The Orlicz setting $X=L^\psi$ is treated first. Under suitable assumptions on the Young function $\psi$ and on the coefficients, the optimal associated space for the lower-order datum is identified. An \textit{a priori} estimate in $W^{1,\psi}$ is then obtained. The discussion is next extended to rearrangement-invariant Banach function spaces. A class $(\mathcal C)$ is introduced to characterize the spaces $X$ for which a corresponding associated space $Y$ yields Schauder-type estimates. Lorentz spaces are finally examined as concrete examples.
	\end{abstract}

	\maketitle
\section{Introduction}\label{sec1}
Let $d \ge 3$, and let $\Omega \subset \mathbb{R}^d$ be a bounded domain with $C^{1}$ boundary. We consider a second-order linear elliptic equation in divergence form, including lower-order terms, of the form
\begin{equation}\label{eq.1}
	\left\{
	\begin{aligned}
		-\operatorname{div}\big(\mathbf A\nabla u+\mathbf B\,u+\mathbf F\big)+\mathbf C\cdot \nabla u+\mathbf V\,u &= g
		& \text{in } \Omega,\\
		u &= 0 & \text{on } \partial \Omega.
	\end{aligned}
	\right.
\end{equation}

Equations of the form \eqref{eq.1} have been the object of a vast literature. They arise naturally in the study of second-order elliptic operators in divergence form and play a fundamental role in analysis, notably in connection with Sobolev spaces, variational methods, and regularity theory. From a probabilistic perspective, such operators are closely related to diffusion processes and Dirichlet forms. In addition, they arise in a wide range of applications, including models of diffusion and transport in physics, as well as in biological systems such as population dynamics and chemotaxis. We refer to standard references for a comprehensive account of these developments \cite{Arumugam2020KellerSegelCM,bass1998diffusions,bogachev1997elliptic,bogachev2001regularity,bogachev2022fokker,gilbarg1998elliptic,morrey1966multiple,nagasawa2012schrodinger}.
\\

We now specify the structural assumptions on the coefficients in \eqref{eq.1}. The leading coefficient $\mathbf A=(\mathbf A^{ij})\colon \Omega\to\mathbb{R}^{d\times d}$ is assumed to be measurable and bounded, and to satisfy the uniform ellipticity condition: there exist constants $m,M>0$ such that
\begin{equation}
	m|\xi|^2\le \mathbf A(x)\,\xi\cdot\xi, \qquad |\mathbf A(x)|\le M
\end{equation}
for all $\xi\in\mathbb{R}^d$ and for almost every $x\in \Omega$. No symmetry of $\mathbf A$ is assumed.

We next impose integrability and local control assumptions on the lower-order coefficients. In the natural energy framework, we assume that $\mathbf B,\mathbf C \in L^{d}(\Omega)$ and $\mathbf V \in L^{d/2}(\Omega)$. In addition, there exist constants $c_0,k_{0}>0$ such that
\[
\int_{B_r(x)\cap \Omega}\bigl(|\mathbf B|^{2}+|\mathbf C|^{2}+|\mathbf V|\bigr)^{d/2}\,dx
\le \bigl(c_0r^{k_0}\bigr)^{d/2}
\]
for every ball $B_r(x)\subset \mathbb{R}^d$. 

Under the above assumptions, we consider the associated shifted variational formulation of \eqref{eq.1}. More precisely, we seek $u\in W^{1,2}_{0}(\Omega)$ such that
\[
\mathcal B(u,v)+\lambda \langle u,v\rangle_{L^2(\Omega)}=\mathcal L(v), \qquad v\in W^{1,2}_{0}(\Omega).
\]
It is well known that this problem satisfies the Fredholm alternative in $W^{1,2}_{0}(\Omega)$. More precisely, for every $\mathbf F\in L^{2}(\Omega)$ and $g\in L^{\frac{2d}{d+2}}(\Omega)$, there exists a discrete set $\mathcal E\subset \mathbb{C}$, without finite accumulation points, such that for each $\lambda\notin \mathcal E$ the variational problem admits a unique solution $u\in W^{1,2}_{0}(\Omega)$. In contrast, for $\lambda\in \mathcal E$, the associated homogeneous problem may possess a finite-dimensional space of nontrivial solutions.

We now turn to regularity issues, which are central to the analysis of \eqref{eq.1}. While the Fredholm framework ensures existence and uniqueness, it does not by itself yield sufficient information on the qualitative properties of solutions. In particular, the case $\lambda=0$, corresponding to \eqref{eq.1}, shows that additional regularity estimates are required, even when $0\notin \mathcal E$ and uniqueness holds in $W^{1,2}_{0}(\Omega)$.

In this direction, one has interior $L^{2}$-estimates for the gradient. Let $D\Subset \Omega$, set $\delta_0=\operatorname{dist}(D,\partial \Omega)$, and assume that $\Omega\subset B_{R}(x_{1})$. Then any solution $u\in W^{1,2}_{0}(\Omega)$ of \eqref{eq.1} satisfies
\[
\|\nabla u\|_{L^2(D)}\le c\Big(\delta_0^{-1}\|u\|_{L^2(\Omega)}+\|\mathbf{F}\|_{L^2(\Omega)}+\|g\|_{L^{\frac{2d}{d+2}}(\Omega)}\Big),
\]
where $c$ depends only on $d$, $m$, $M$, $c_0$, $k_0$, and $R$. Under stronger assumptions on the coefficients, one can further obtain $W^{2,2}$-regularity, both locally in $\Omega$ and globally up to the boundary. We refer to \cite{morrey1966multiple} for further details.

In this paper, we investigate the following regularity problem for \eqref{eq.1}. Let $X(\Omega)\subset L^{2}(\Omega)$ be a given data space for $\mathbf F$. Under suitable assumptions on the coefficients $\mathbf A$, $\mathbf B$, $\mathbf C$, and $\mathbf V$, we aim to identify the optimal space $Y(\Omega)$ associated with $X(\Omega)$ such that, for every $\mathbf F\in X(\Omega)$ and $g\in Y(\Omega)$, any weak solution $u$ to \eqref{eq.1} belongs to $W^{1}X(\Omega)$, where
\[
W^{1}X(\Omega)=\{v\in X(\Omega)\,:\,\nabla v\in X(\Omega)\},
\]
and satisfies the estimate
\begin{equation}\label{est.1}
	\|u\|_{W^{1} X(\Omega)}\lesssim \|u\|_{L^1(\Omega)}+\|\mathbf{F}\|_{X(\Omega)}+\|g\|_{Y(\Omega)}.
\end{equation}

In the case of $X=L^{q}(\Omega)$, the associated space is $Y=L^{p}(\Omega)$, where $p=\frac{dq}{d+q}$. Assume that $\mathbf A$ is continuous on an open set $\mathcal O\supset \overline{\Omega}$, and that $\mathbf B$, $\mathbf C$, and $\mathbf V$ are measurable with
\[
\mathbf B \in L^{q}(\mathcal O), \qquad \mathbf C \in L^{d}(\mathcal O), \qquad \mathbf V \in L^{p}(\mathcal O).
\]
Then, for every $\mathbf F \in L^{q}(\Omega)$ and $g\in L^{p}(\Omega)$, any weak solution $u$ to \eqref{eq.1} belongs to $W^{1,q}(\Omega)$ and satisfies
\[
\|u\|_{W^{1,q}(\Omega)}\le c\bigl(\|u\|_{L^1(\Omega)}+\|\mathbf{F}\|_{L^q(\Omega)}+\|g\|_{L^p(\Omega)}\bigr),
\]
where $c$ is independent of $u$, $\mathbf F$, and $g$; see \cite{morrey1966multiple}.

The Lebesgue scale is not sufficient, in general, to determine optimal data spaces \(X\). In critical regimes, the associated space \(Y\) may not belong to the Lebesgue scale and may require a modified growth condition. This phenomenon appears, for example, in the reduced Keller--Segel model. In the supercritical regime, the Hölder regularity obtained in the Lebesgue setting is sufficient, but this is no longer true in the critical case. Orlicz spaces provide an appropriate framework in this situation, since they allow the required correction of the growth condition while preserving the functional properties needed in the analysis; see \cite{musil2023optimality}.

We work under the following assumptions. Let $\psi$ be an $N$-function such that $\psi\in \Delta_{2}\cap \nabla_{2}$ and
\[
\mathcal{I}_{\psi,1}^{\infty}:=\int^{\infty}\left(\frac{t}{\psi(t)}\right)^{1/(d-1)}dt<\infty.
\]
We then take $X=L^{\psi}(\Omega)$ as the data space for $\mathbf F$; see Section~\ref{sec3} for further details. We aim to identify the space $Y(\Omega)$ associated with $X$ such that, for $\mathbf F\in X$ and $g\in Y(\Omega)$, any weak solution $u$ to \eqref{eq.1} satisfies the \emph{a priori} estimate \eqref{est.1} in $W^{1,\psi}(\Omega)$. In this setting, the estimate yields continuity of $u$ with a modulus of continuity determined by $\psi$. We shall refer to such bounds as \textit{Schauder--Orlicz--type estimates}. For related results on gradient estimates in Orlicz spaces for nonlinear elliptic equations, see the following references \cite{byun2008gradient,byun2011gradient}.
 
We say that the coefficients $\mathbf A$, $\mathbf B$, $\mathbf C$, and $\mathbf V$ satisfy the $\mathcal H_{\psi}$-condition on a domain $\mathcal O\subset\RR^d$ if $\mathbf A\in C(\mathcal O)$ and the functions $\mathbf B$, $\mathbf C$, and $\mathbf V$ are measurable with
\[
\mathbf B \in L^{\psi}(\mathcal O), \qquad \mathbf C \in L^{d}(\mathcal O), \qquad \mathbf V \in L^{\gamma}(\mathcal O),
\]
where $\gamma = [\psi_d^*]^*$.

The relation between $\psi$ and $\gamma$ can be expressed in terms of their inverse functions. Since $\gamma=[\psi_d^*]^*$, one has
\[
\gamma^{-1}(t)\sim \frac{t}{[\psi_d^*]^{-1}(t)}.
\]
Moreover
\[
[\psi_d^*]^{-1}(t)\sim t^{-1/d}[\psi^*]^{-1}(t),
\]
and hence
\[
\gamma^{-1}(t)\sim \frac{t^{1+1/d}}{[\psi^*]^{-1}(t)}.
\]
Using $[\psi^*]^{-1}(t)\sim \frac{t}{\psi^{-1}(t)}$, it follows that
\[
\gamma^{-1}(t)\sim t^{1/d}\psi^{-1}(t).
\]
In particular, one obtains $\gamma_d\sim \psi$.

We now state the main result of this paper.
\begin{theorem}\label{main theorem}
	Let $d\ge 3$. Let $\Omega \subset \mathbb{R}^d$ be a bounded domain of class $C^{1}$, and let $\mathcal O$ be an open set such that $\Omega \subset \mathcal O$. Let $\psi$ be an $N$-function satisfying $\psi\in \Delta_{2}\cap \nabla_{2}$ and $\mathcal{I}_{\psi,1}^{\infty}<\infty$. Assume that the coefficients satisfy the $\mathcal{H}_{\psi}$-condition on $\mathcal O$. 
	
	If $\mathbf F \in L^{\psi}(\Omega)$ and $g \in L^{\gamma}(\Omega)$, where $\gamma=[\psi_d^*]^*$, then any weak solution $u$ to \eqref{eq.1} in $\Omega$ satisfies
	\begin{equation}\label{est:main}
		\|u\|_{W^{1,\psi}(\Omega)} \le c \Big( \|u\|_{L^1(\Omega)} + \|\mathbf{F}\|_{L^\psi(\Omega)} + \|g\|_{L^\gamma(\Omega)} \Big),
	\end{equation}
	where $c>0$ depends only on $d$, $\Omega$, $m$, $M$, $\omega_{\mathbf A}$, $\psi$, and $\Lambda$, and is independent of $u$, $\mathbf F$, and $g$. Here $\omega_{\mathbf A}$ denotes the modulus of continuity of $\mathbf A$, and $\Lambda$ depends only on the bounds of the lower-order coefficients.
\end{theorem}
The proof of the main result is based on the method developed in \cite{morrey1966multiple}. The local estimates follow from the boundedness of the Calderón--Zygmund and Riesz operators, while the global estimate in $\Omega$ is obtained through a localization procedure. We also point out that the techniques of \cite{byun2008elliptic,byun2008gradient,byun2011gradient}, which strongly depend on the Reifenberg geometry, are not directly applicable in our setting.\\

\noindent The plan of the paper is the following. Section~\ref{sec3} contains the preliminaries on \(N\)-functions, Orlicz spaces, Orlicz--Sobolev spaces, and the embedding results used in the paper. Section~\ref{sec4} is devoted to the proof of Theorem~\ref{main theorem}. We first prove local estimates on balls and half-balls and then derive the corresponding estimate in \(\Omega\). Section~\ref{sec5} presents examples of admissible functions \(\psi\). In Section~\ref{sec6}, we introduce a class \((\mathcal C)\) of Banach spaces \(X\) for which estimate \eqref{est.1} holds. In the case \(X=L^{\psi}\), we use the results of \cite{cianchi1999strong} to identify the associated space \(Y\). Finally, we consider Lorentz spaces and investigate whether they belong to the class \((\mathcal C)\).

\section{Preliminaries}\label{sec3}
Throughout this paper we use the notation and basic facts from \cite{rao1991theory, cianchi2009orlicz, harjulehto2019generalized}. Let \(d \geq 3\), and let \(\Omega\) be an open subset of \(\mathbb{R}^d\). Let \(\psi:[0,\infty)\to[0,\infty)\) be an \textit{\(N\)-function}; namely, \(\psi\) is convex and increasing, \(\psi(0)=0\), and
\[
\lim_{t\to 0^+}\frac{\psi(t)}{t}=0,
\qquad
\lim_{t\to \infty}\frac{\psi(t)}{t}=\infty .
\]
Let \(\psi_1\) and \(\psi_2\) be \(N\)-functions. We write
\[
\psi_1 \prec \psi_2,
\]
if there exists \(c>0\) such that
\[
\psi_1(t)\le \psi_2(ct)
\qquad \text{for all } t>0.
\]
If both \(\psi_1\prec \psi_2\) and \(\psi_2\prec \psi_1\) hold, then we write
\[
\psi_1\sim \psi_2.
\]
We say that \(\psi\) satisfies the \textit{\(\Delta_2\)-condition} if there exists \(\ell>0\) such that
\[
\psi(2t)\le \ell\,\psi(t)
\qquad \text{for all } t\ge0.
\]
The \textit{conjugate function} \(\psi^*\) of \(\psi\) is defined by
\[
\psi^*(s)=\sup_{t\ge0}\bigl\{st-\psi(t)\bigr\},
\qquad s\ge0.
\]
The function \(\psi^*\) is again an \(N\)-function. We say that \(\psi\) satisfies the \textit{\(\nabla_2\)-condition} if \(\psi^*\) satisfies the \(\Delta_2\)-condition.

For example, the pair of complementary \(N\)-functions
\[
\psi(t)=(1+t)\log(1+t)-t,
\qquad 
\psi^*(s)=e^s-s-1,
\]
provides a typical situation where the \(\Delta_2\)- and \(\nabla_2\)-conditions are not simultaneously satisfied, since
\[
\psi\in \Delta_2\setminus \nabla_2,
\qquad 
\psi^*\in \nabla_2\setminus \Delta_2.
\]

Assume that \(\psi\) satisfies the \(\Delta_2\cap\nabla_2\)-conditions . Then there exist constants \(1<p_\psi^-\le p_\psi^+<\infty\) and \(L\ge1\) such that for every \(\alpha>1\) and every \(t>0\),
\begin{equation}\label{eq.HQ}
	L^{-1}\alpha^{p_\psi^-}\psi(t)\le \psi(\alpha t)\le L\alpha ^{p_\psi^+}\psi(t).
\end{equation}
Here \(p_\psi^-\) and \(p_\psi^+\) denote the Matuszewska--Orlicz indices of \(\psi\). We denote by \(\psi^{-1}\) the left inverse of \(\psi\), defined by \(\psi^{-1}(s):=\inf\{t\ge0:\psi(t)>s\}\). Since \(\psi\) is convex and satisfies \((\mathrm{aDec})_{p_\psi^+}\), that is, there exists a constant \(L\ge1\) such that \(\psi(t)/t^{p_\psi^+}\le L\,\psi(s)/s^{p_\psi^+}\) for all \(0<s<t\), it follows from \cite[Corollary~2.3.4, p.~25]{harjulehto2019generalized} that \(\psi\) is bijective, \(\psi^{-1}\) coincides with the usual inverse function, and \(\psi^{-1}(t)\,[\psi^*]^{-1}(t)\sim t\).

We denote by \(L^\psi(\Omega)\) the Orlicz space of all measurable functions \(u:\Omega\to\mathbb R\) such that
\[
\int_{\Omega} \psi\left(\frac{|u(x)|}{\lambda}\right)\,dx<\infty
\qquad \text{for some } \lambda>0.
\]
Equipped with the \emph{Luxemburg norm} \cite{luxemburg1955banach}
\[
\|u\|_{L^\psi(\Omega)}:=\inf\left\{\lambda>0:\int_{\Omega} \psi\left(\frac{|u(x)|}{\lambda}\right)\,dx\le 1\right\}.
\]
Then \(L^\psi(\Omega)\) is a Banach space. If \(\psi\) satisfies  \(\Delta_2\)-condition, then \(L^\psi(\Omega)\) coincides with the Orlicz class 
$$K^\psi(\Omega)
= \left\{\, u \,: \
\int_{\Omega} \psi(|u(x)|)\,dx < \infty \right\}.$$
If \(\psi\) satisfies the \(\Delta_2\cap \nabla_2\)-condition, then
\[
(L^\psi(\Omega))^*=L^{\psi^*}(\Omega),
\]
and \(L^\psi(\Omega)\) is reflexive.

Let \(m\in \mathbb{N}\). The Orlicz--Sobolev space \(W^{m,\psi}(\Omega)\) consists of all functions \(u\in L^\psi(\Omega)\) such that
\[
D^\alpha u\in L^\psi(\Omega)
\]
for every multi-index \(\alpha\) satisfying \(|\alpha|\le m\). Here, for
\[
\alpha=(\alpha_1,\dots,\alpha_d),
\]
we set
\[
|\alpha|=\alpha_1+\cdots+\alpha_d,
\]
and
\[
D^\alpha u
=
\frac{\partial^{|\alpha|}u}
{\partial x_1^{\alpha_1}\cdots \partial x_d^{\alpha_d}}.
\]
The space \(W^{m,\psi}(\Omega)\) is a Banach space when equipped with the norm
\[
\|u\|_{W^{m,\psi}(\Omega)}
:=
\sum_{|\alpha|\le m}
\|D^\alpha u\|_{L^\psi(\Omega)}.
\]

Moreover if \(\psi\) satisfies the \(\Delta_2\cap\nabla_2\)-condition, then \(W^{m,\psi}(\Omega)\) is reflexive. Further details can be found in \cite{harjulehto2019generalized}.

Let \(\psi\) be an \(N\)-function satisfying  \(\Delta_2\cap\nabla_2\)-condition and
\begin{equation}\label{eq:32}
	\mathcal I_{\psi,1}^0:=\int_{0}^s \left(\frac{t}{\psi(t)}\right)^{1/(d-1)} dt < \infty .
\end{equation}
Define \(H:[0,\infty)\to[0,\infty)\) by $H(s)=\left[\mathcal I_{\psi,1}^0\right]^{(d-1)/d},$ and set \(\psi_d(t):=\psi\bigl(H^{-1}(t)\bigr)\) for \(t\ge0\), where \(H^{-1}\) denotes the generalized left inverse of \(H\). Then by \cite[Theorem~3.1%, p.~86 
]{cianchi2009orlicz}, there exists a constant \(c>0\) such that \(\|u\|_{L^{\psi_d}(\Omega)}\le c\,\|u\|_{W^{1,\psi}(\Omega)}\) for all \(u\in W^{1,\psi}(\Omega)\). Moreover \(L^{\psi_d}(\Omega)\) is optimal among Orlicz target spaces for this embedding, in the sense that \(W^{1,\psi}(\Omega)\hookrightarrow L^\phi(\Omega)\) implies \(L^{\psi_d}(\Omega)\hookrightarrow L^\phi(\Omega)\) for every Young function \(\phi\).\\

 A function $\varpi:[0,\infty)\to[0,\infty)$ is called a modulus of continuity if it is increasing and vanishes
at 0.  We denote by $C^\varpi(\Omega)$ the space of all uniformly continuous functions $u$ on $\Omega$ whose modulus of continuity is bounded by $\varpi$.
The space $C^\varpi(\Omega)$ is endowed with the norm
\[
\|u\|_{C^\varpi(\Omega)}
:=
\|u\|_{C(\Omega)}
+
\sup_{x\ne y}\frac{|u(x)-u(y)|}{\varpi(|x-y|)},
\]
where \(\|u\|_{C(\Omega)}=\sup_{x\in \Omega}|u(x)|\).

 Typical examples are given by \(\varpi(r)=r^\alpha\) with \(0<\alpha\le1\), which yields the Hölder space \(C^{0,\alpha}(\Omega)\), and by \(\varpi(r)=r\), which yields the space of Lipschitz continuous functions. Define \(\Theta_{\psi}:(0,\infty)\to[0,\infty]\) by
\[
\Theta_{\psi}(t):= t^{d/(d-1)}\int_t^{\infty}\frac{\psi^*(s)}{s^{1+d/(d-1)}}\,ds,
\qquad t>0.
\]
Assume that
\begin{equation}\label{eq:44}
	\mathcal I_{\psi,1}^\infty:=\int^\infty \left(\frac{t}{\psi(t)}\right)^{1/(d-1)}\,dt<\infty.
\end{equation}
Then \(\varpi_\psi(r):=r^{1-d}/\Theta_{\psi}^{-1}(r^{-d})\) is a modulus of continuity, and
$
W^{1,\psi}(\Omega)\hookrightarrow C^{\varpi_\psi}(\Omega);
$
see \cite[Theorem~4.2, p.~93]{cianchi2009orlicz}.
\vspace{0.3 cm}

Typical examples of \(\psi\) are presented in Table \ref{tab:psi}.

\begin{small}
\begin{table}[h]\label{Table1}
	\centering
	\renewcommand{\arraystretch}{1.3}
	\begin{small}
	\begin{tabular}{|l|l|l|l|l}
		\hline
		$\psi(t)$ 
		& $t^q,\, q>1$ 
		&  $\sim_{\infty} t^{q}(\log t)^{\alpha},\, q>1, \, \alpha \in\RR$ 
		& $\sim_{\infty} t^{q}(\log\log t)^{\alpha},\, q>1, \, \alpha \in\RR$ \\
		\hline
		$\psi^*(t)$
		& $t^{q/(q-1)}$
		&  $\sim_\infty t^{q/(q-1)} (\log  t)^{-\alpha/(q-1)}$  
		& $\sim_\infty t^{q/(q-1)} (\log \log  t)^{-\alpha/(q-1)}$ \\
		\hline
		\\
		$\psi_d(t)$ 
		& $ t^{d q/(d-q)},\, q< d$
		&  $\sim_\infty 
		t^{dq/(d-q)}(\log t)^{\alpha d /(d-q)},\, q<d$
		& $\sim_\infty 
		t^{qd/(d-q)}(\log\log t)^{\alpha d /(d-q)}, \, q<d$ \\
		& 
		&  $\sim_\infty 
		\exp(\,t^{d/(d-1-\alpha)}),\, q=d,\, \alpha< d-1$
		& $ \sim_{\infty}
		\exp(\,t^{d/(d-1) (\log t)^{\alpha/(d-1)}}),\, q=d$ \\
		& 
		& $\sim_\infty \exp\!\big(\exp(t^{d/(d-1)})\big),\, q=d,\, \alpha=d-1$
		& \\
		\\
		\hline 
		\\
		$\varpi_{\psi}(t)$ 
		& $t^{1-d/q},\, q>d $
		&  $\sim_\infty t^{1-d/q} (-\log t )^{-\alpha/q}, \, q>d$
		& $\sim_\infty  
		t^{\,1-d/q}
		\big(\log(-\log t)\big)^{-\alpha/q},\, q>d$ \\
		&  &$\sim_\infty (- \log t)^{-(\alpha-(d-1))/d},\, q=d,\, \alpha>d-1$ & \\
		\hline  
	\end{tabular}
\end{small}
	\vspace{0.15cm}
	\caption{Typical examples of \(\psi\) and the associated functions \(\psi^*\), \(\psi_d\), and \(\varpi_\psi\).}
	\label{tab:psi} 
\end{table}
\end{small}

\section{Proof of Theorem \ref{main theorem}}\label{sec4}
%The proof begins with a localization argument near an arbitrary point \(x_{0}\in \overline \Omega \). 
%After translation in the interior case, the local domain is reduced to a ball \(B_{2r}\). Near the boundary, a \(C^{1}\)-chart flattens \(\partial G\) and reduces the geometry to a half-ball \(B^{+}_{2r}\) with flat boundary $\sigma_{2R}$. It therefore suffices to work on a local domain $G=B_{2r}$ or $G=B^{+}_{2r}$, with $R$ chosen small enough.\\
%For $x = (x_1, \dots, x_{d-1},x_d)= (x', x_d) \in \mathbb{R}^d$, we write $B_r(x) = \{y \in \mathbb{R}^d : |x-y| < r\}$, $B_r = B_r(0)$, $B^+_r =  B_r \cap \{x_d > 0\}$, and $T_r = B_r \cap \{x_d = 0\}$.
%Let $\eta$ be a non increasing $C^{\infty}$ cut-off function on $\RR$ such that $\eta(s)=1$ for $s\le \frac54$ and $\eta(s)=0$ for $s\ge \frac74$. 

The proof begins with a localization argument in a neighborhood of an arbitrary point 
\(x_{0}\in \overline{\Omega}\). For 
\(x = (x_1, \dots, x_{d-1}, x_d)= (x', x_d) \in \mathbb{R}^d\), we denote by
$$
B_r(x) = \{y \in \mathbb{R}^d : |x-y| < r\}
$$
with \(B_r := B_r(0)\). We also denote by
$$
B_r^{+} := B_r \cap \{x_d > 0\},
\qquad
T_r: = B_r \cap \{x_d = 0\}.
$$
Let $\eta \in C^\infty(\mathbb{R}_{+})$ be a non-increasing cut-off function such that
$$
\eta(s)=
\left\{
\begin{array}{ll}
	1, & \text{for } s \leq \frac{5}{4},\\[0.3em]
	0, & \text{for } s \geq \frac{7}{4},
\end{array}
\right.$$
and $$ \eta_r(x)=\eta (|x|/r),
\quad x\in \Omega.
$$
The modified coefficients are defined by
$$
\mathbf{A}_r^{ij}(x)
=
\bigl(1-\eta_r(x)\bigr)\mathbf{A}_0^{ij}
+
\eta_r(x)\mathbf{A}^{ij}(x),
\qquad
\mathbf{A}_0=\mathbf{A}(0),
$$
and
$$
\mathbf{B}_r^i(x)=\eta_r(x)\mathbf{B}^i(x),
\qquad
\mathbf{C}_r^i(x)=\eta_r(x)\mathbf{C}^i(x),
\qquad
\mathbf{V}_r(x)=\eta_r(x)\mathbf{V}(x).
$$
Assume that $u$ is a weak solution to \eqref{eq.1} in $\Omega$, with
$$
\operatorname{supp} u \subset B_r
\quad \text{or} \quad
\operatorname{supp} u \subset B_r^+ \cup T_r .
$$
Then $u$ satisfies
\begin{equation}\label{form.1}
	\int_{\Omega} \left\{ \sum_{i=1}^d \partial_{i}v \left( \sum_{j=1}^d\mathbf{A}_r^{ij} \partial_{j}u + \mathbf{B}_r^{i} u + \mathbf{F}^{i} \right) + v \left( \sum_{i=1}^d\mathbf{C}_r^{i} \partial_{i}u + \mathbf{V}_r u + g \right) \right\} dx = 0. 
\end{equation}
Since $\mathbf{A}^{i j}_r= \mathbf{A}^{i j}_0+(\mathbf{A}^{i j}_r-\mathbf{A}^{i j}_0)$, we then rewrite \ref{form.1} in the form
\begin{align*}
	\sum_{i,j=1}^d\int_{\Omega} \partial_{i}v\, \mathbf{A}^{i j}_0\, \partial_{j} u\, dx
	&= - \sum_{i=1}^d \int_{\Omega} \partial_{i}v\left(\sum_{j=1}^d (\mathbf{A}^{i j}_r -\mathbf{A}^{i j}_0)\, \partial_{j} u
	+ \mathbf{B}_r^{i}\, u + \mathbf{F}^{i} \right)\, dx\\
	& \qquad \qquad \qquad \qquad \qquad \qquad - \int_{\Omega} v\left( \sum_{i=1}^d\mathbf{C}_r\, \partial_{i}u +  \mathbf{V}_r\, u + g \right)\, dx.
\end{align*}
In other words, $u$ is a weak solution to
\begin{align*}
	- \operatorname{div}(\mathbf{A}_0 \nabla u)
	&= \operatorname{div}\bigl[(\mathbf{A}_r - \mathbf{A}_0)\nabla u + \mathbf{B}_r u + \mathbf{F}\bigr]
	- \bigl(\mathbf{C}_r \cdot \nabla u + \mathbf{V}_r u + g\bigr)
	\quad \text{in } \Omega \\
	&:= \operatorname{div}[\mathbf{G}] - \mathbf{h}.
\end{align*}

Let $\mathcal Q_{2r}[\mathbf{G}]$ and $\mathcal P_{2r}[\mathbf{h}]$ denote, respectively, the quasi-potential and the potential associated with the operator $-\operatorname{div}(\mathbf{A}_0\nabla\cdot)$ in $\Omega$. They are defined as the unique functions in $W^{1,\psi}_{0}(\Omega)$ satisfying
\[
-\operatorname{div}(\mathbf A_0\nabla \mathcal Q_{2r}[\mathbf{G}])=\operatorname{div}(\mathbf{G})
\]
and
\[
-\operatorname{div}(\mathbf A_0\nabla \mathcal P_{2r}[\mathbf{h}])=\mathbf{h}.
\]
We define
\[
u_r:=\mathcal Q_{2r}[\mathbf{G}]-\mathcal P_{2r}[\mathbf{h}].
\]
By construction, $u_r$ satisfies
\[
-\operatorname{div}(\mathbf A_0\nabla u_r)
= \operatorname{div}[\mathbf{G}] - \mathbf{h}
\quad \text{in } \Omega,
\]
so the difference $\mathbf{H}_r := u - u_r$ solves
\[
-\operatorname{div}(\mathbf A_0\nabla \mathbf{H}_r)=0
\quad \text{in } \Omega.
\]
Since $d>2$, we have
\[
\mathcal Q_{2r}[\mathbf{G}](x), \,
\mathcal P_{2r}[\mathbf{h}](x) \to 0
\quad \text{as } |x|\to\infty,
\]
and, together with $\operatorname{supp} u \subset B_r$, this implies that
$\mathbf{H}_r \equiv 0$. Therefore
\begin{equation}\label{eq:local-TR}
	\begin{aligned}
		u=u_r
		&= \mathcal Q_{2r}\Big[(\mathbf{A}_r - \mathbf A_0)\cdot \nabla u_r + \mathbf{B}_r u_r\Big]
		- \mathcal P_{2r}\Big[\mathbf{C}_r \cdot \nabla u_r + \mathbf{V}_r u_r\Big] \\
		&\quad + \mathcal Q_{2r}[\mathbf{F}]
		- \mathcal P_{2r}[g] \\
		&:= \mathcal K_r u_r + \mathbf z_r .
	\end{aligned}
\end{equation}

\begin{proposition}\label{prop:operators}
	Let $\psi$ be an $N$-function satisfying the $\Delta_2 \cap \nabla_2$ condition with $\gamma=[\psi_d^*]^*$, and let $r>0$. Then
	\begin{enumerate}
		\item [(1)] The operator $\mathcal{Q}_{2r}$ is bounded from $L^\psi(B_{2r})$ into $W^{1,\psi}(B_{2r})$.
		
		\item [(2)] If $\mathcal I^0_{\gamma,1}<\infty$, then $\mathcal{P}_{2r}$ is bounded from $L^\gamma(B_{2r})$ into $W^{1,\psi}(B_{2r})$.
	\end{enumerate}
	Moreover the corresponding bounds are independent of $r$.
\end{proposition}

\vspace{0.2cm}

The following lemmas will be needed in the proof below.

\vspace{0.2cm}

\begin{lemma}\label{lem:1}
	Let $\psi$ be an $N$-function satisfying the $\Delta_2 \cap \nabla_2$ condition, and let $\mathcal{S}$ be a Calderón--Zygmund singular integral operator. Then
	\[
	\|\mathcal{S}f\|_{L^{\psi}(\RR^d)} \lesssim \|f\|_{L^{\psi}(\RR^d)}.
	\]
\end{lemma}

\begin{proof}
	Since $\psi$ is independent of $x$, conditions $(\mathrm{A}0)$, $(\mathrm{A}1)$, and $(\mathrm{A}2)$ are satisfied. Moreover the assumption $\psi\in\Delta_2\cap\nabla_2$ implies $(\mathrm{aInc})_{p_\psi^-}$ and $(\mathrm{aDec})_{p_\psi^+}$. Therefore the result follows from \cite[Corollary 5.4.3, p.~116]{harjulehto2019generalized}.
\end{proof}

\begin{lemma}\label{lem:2}
	Let $\psi$ be an $N$-function satisfying the $\Delta_2 \cap \nabla_2$ condition, and let $f\in L^\psi(B_{2r})$. Then the Dirichlet problem $\Delta w=f$ admits a unique strong solution
	\[
	w\in W^{2,\psi}(B_{2r})\cap W^{1,\psi}_0(B_{2r})
	\]
	such that
	\[
	\|w\|_{W^{2,\psi}(B_{2r})} \lesssim \|f\|_{L^\psi(B_{2r})}.
	\]
\end{lemma}

\begin{proof}
	The result follows immediately from \cite[Theorem~1.2]{hasto2019calderon}.
\end{proof}

\vspace{0.2cm}

\begin{proof}[Proof of Proposition \ref{prop:operators}]
	Under the bilinear form, the antisymmetric part of $\mathbf A_0$ vanishes, while the symmetric part is positive definite by uniform ellipticity. A fixed invertible linear change of variables reduces $\mathbf A_0$ to the identity matrix and maps $B_{2r}$ onto an ellipsoid contained in the interior of $\Omega$. The transformed functions are still denoted by $u$ and $v$, so that
	$$
	\sum_{i,j=1}^{d}\int_{\Omega} \partial_{i}v\,\mathbf A_0^{ij}\,\partial_{j} u\,dx
	=
	\sum_{i=1}^{d}\int_{\Omega} \partial_{i}v\,\partial_{i}u\,dx.
	$$

	To prove $(1)$, observe that  the quasi-potential $\mathcal Q_{2r}[\mathbf{G}]$ is represented in terms of the first derivatives of the Green kernel extended by zero outside $B_{2r}$. The scaled norm adapted to $B_{2r}$ is defined by
		$$'\|\cdot\|_{W^{1,\psi}(B_{2r})}:=\|\nabla(\cdot)\|_{L^\psi(B_{2r})}+(2r)^{-1}\|\cdot\|_{L^\psi (B_{2r})}.   $$
		
		Since $\mathcal Q_{2r}[\mathbf{G}] \in W^{1,\psi}_0(B_{2r})$, Poincar\'e's inequality reduces the estimate to the gradient term. For $j \in \{1,\dots,d\}$, we have
		$$
		\partial_{j}\mathcal Q_{2r}[\mathbf{G}]
		= - \sum_{i=1}^d \int_{\mathbb{R}^d} \partial_{ij} \widetilde{\mathbf{k}}_{0}(x-y)\, \mathbf{G}^{i}(y)\,dy,
		$$
		where
		$$
		\partial_{i} \widetilde{\mathbf{k}}_{0}(z)= -\omega_d^{-1} z^i |z|^{-d},
		$$
		and
		$$
		\partial_{ij}\widetilde{\mathbf{k}}_{0}(z)
		= -\omega_d^{-1}\Big(\delta_{ij} |z|^{-d} - d\, z^i z^j |z|^{-(d+2)}\Big),
		\qquad z\neq 0.
		$$
		
		Since $|z^i z^j| \le |z|^2$, the kernel $\partial_{ij}\widetilde{\mathbf{k}}_{0}$ is homogeneous of degree $-d$. Moreover its regularity away from the origin implies
		$$
		|\partial_{ij}\widetilde{\mathbf{k}}_{0}(z+h)-\partial_{ij}\widetilde{\mathbf{k}}_{0}(z)|
		\lesssim |h|\,|z|^{-d-1}
		\quad \text{whenever } 2|h|\le |z|,
		$$
		and consequently, for every $0<\varepsilon\le 1$,
		$$
		|\partial_{ij}\widetilde{\mathbf{k}}_{0}(z+h)-\partial_{ij}\widetilde{\mathbf{k}}_{0}(z)|
		\lesssim |h|^\varepsilon |z|^{-d-\varepsilon}.
		$$
		Thus $\partial_{ij}\widetilde{\mathbf{k}}_{0}$ is a Calder\'on--Zygmund kernel, and Lemma~\ref{lem:1} yields
		$$
		'\|\mathcal Q_{2r}[\mathbf{G}]\|_{W^{1,\psi}(B_{2r})} \lesssim \|\mathbf{G}\|_{L^\psi (B_{2r})}.
		$$

		To prove $(2)$, let $\mathbf{h}\in L^\gamma(B_{2r})$. By Lemma~\ref{lem:2}, we have
		$\mathcal{P}_{2r}(\mathbf{h})\in W^{2,\gamma}(B_{2r})$ and
		\[
		\|\nabla^2 \mathcal{P}_{2r}(\mathbf{h})\|_{L^\gamma(B_{2r})}
		\lesssim \|\mathbf{h}\|_{L^\gamma(B_{2r})}.
		\]
		If $\mathcal{I}_{\gamma,1}^0<\infty$, then the embedding
		$W^{1,\gamma}(B_{2r}) \hookrightarrow L^\psi(B_{2r})$ holds. Hence
		$$
		'\|\mathcal {P}_{2r}(\mathbf{h})\|_{W^{1,\psi}(B_{2r})}\lesssim \, '\|\mathcal {P}_{2r}(\mathbf{h})\|_{W^{2,\gamma}(B_{2r})}\lesssim \|\mathbf{h}\|_{L^\gamma (B_{2r})}.
		$$
\end{proof}

This yields the following result.
\begin{proposition}\label{Thm:TR bound}
	Let $\psi$ be an $N$-function satisfying the $\Delta_2 \cap \nabla_2$-condition, and let $R>0$. Assume that the coefficients satisfy the $\mathcal{H}_{\psi}$-condition on $B_R$. Then, for every $0<r \leq R/2$, the operator $\mathcal{K}_r$ is bounded on $W^{1,\psi}(B_{2r})$. Moreover its operator norm on $W^{1,\psi}(B_{2r})$ satisfies
	\[
	\|\mathcal{K}_r\| \leq \frac12
	\]
	for sufficiently small $r$.
\end{proposition}

%%%%%%%%%%%%%%%%%%%%%%%%%%%%%%%%%%%%%%%%%%%%%%%%%%%%%
\begin{proof}
	By the triangle inequality and Proposition \ref{prop:operators}, we have
	\begin{align*}
		'\|\mathcal{K}_r u_r\|_{W^{1,\psi} (B_{2r})}
		&\le
		\, '\|\mathcal Q_{2r}[\mathbf{G}]\|_{W^{1,\psi} (B_{2r})}
		+\, 
		'\|\mathcal {P}_{2r}[\mathbf{h}]\|_{W^{1,\psi} (B_{2r})}\\
		&\lesssim
		\|\mathbf{G}\|_{L^\psi (B_{2r})}
		+
		\|\mathbf{h}\|_{L^\gamma (B_{2r})}.
	\end{align*}
	We first estimate \(\|\mathbf{G}\|_{L^\psi (B_{2r})}\). By H\"older's inequality in Orlicz spaces (see, e.g., \cite{o1960fractional}), we obtain
	\begin{align*}
		\|\mathbf{G}\|_{L^\psi (B_{2r})} &\leq \|(\mathbf{A}_r-\mathbf A_0)\cdot \nabla u_r\|_{L^\psi (B_{2r})}
		+\|\mathbf{B}_r u_r\|_{L^\psi (B_{2r})}\\
		&\leq \|\mathbf{A}_r-\mathbf A_0\|_{\infty, B_{2r}} \, \|\nabla u_r\|_{L^\psi (B_{2r})} + \|\mathbf{B}_r\|_{L^\psi (B_{2r})}\, \|u_r\|_{\infty, B_{2r}}\\
		& \lesssim \Big(\|\mathbf{A}_R-\mathbf A_0\|_{\infty, B_{2r}} +\|\mathbf{B}_r\|_{L^\psi (B_{2r})}\, r\psi^{-1}(r^{-d})\Big)  \, '\|u_r\|_{W^{1,\psi}(B_{2r})}\\
		& \le \|\eta_r\|_{\infty, B_{2r}} \Big(\|\mathbf{A}-\mathbf A_0\|_{\infty, B_{2r}} +\|\mathbf{B}\|_{L^\psi (B_{2r})}\, r \psi^{-1}(r^{-d})\Big)  \, '\|u_r\|_{W^{1,\psi}(B_{2r})},
	\end{align*}
	and also
	\begin{align*}
		\|\mathbf{h}\|_{L^\gamma (B_{2r})} &\leq \|\mathbf{C}_r\cdot\nabla u_r\|_{L^\gamma (B_{2r})}
		+\|\mathbf{V}_r u_r\|_{L^\gamma (B_{2r})}\\
		&\leq  \|\mathbf{C}_r\|_{L^d (B_{2r})}\,  \|\nabla u_r\|_{L^{\gamma_1} (B_{2r})}
		+\|\mathbf{V}_r\|_{L^\gamma (B_{2r})}\| u_r\|_{\infty, B_{2r}},
	\end{align*}
	where \(\gamma^{-1}(t)\sim t^{1/d}\gamma_1^{-1}(t)\), and \(\gamma_1 \sim \psi\). Therefore
	\begin{align*}
		\|\mathbf{h}\|_{L^\gamma (B_{2r})} &\leq  \|\mathbf{C}_r\|_{L^d (B_{2r})}\,  \|\nabla u_r\|_{L^\psi (B_{2r})}
		+\|\mathbf{V}_r\|_{L^\gamma (B_{2r})} \, r\psi^{-1}(r^{-d})\,  '\| u_r\|_{L^\psi (B_{2r})}\\
		& \lesssim \Big(\|\mathbf{C}_r\|_{L^d (B_{2r})} + \|\mathbf{V}_r\|_{L^\gamma (B_{2r})} \, r \psi^{-1}(r^{-d}) \Big)\,  '\| u_r\|_{W^{1,\psi} (B_{2r})}\\
		& \leq \|\eta_r\|_{\infty, B_{2r}}\Big(\|\mathbf{C}\|_{L^d (B_{2r})} + \|\mathbf{V}\|_{L^\gamma (B_{2r})} \, r \psi^{-1}(r^{-d}) \Big)\,  '\| u_r\|_{W^{1,\psi} (B_{2r})}.
	\end{align*}
	Combining the above estimates, we obtain
	\begin{align*}
		'\|\mathcal{K}_r u_r\|_{W^{1,\psi} (B_{2r})}
		&\lesssim \,   \|\eta_r\|_{\infty, B_{2r}}\,\Big(\|\mathbf{A}-\mathbf A_0\|_{\infty, B_{2r}} +\|\mathbf{B}\|_{L^\psi (B_{2r})}\, r \psi^{-1}(r^{-d})\\
		& \qquad \qquad \qquad +\|\mathbf{C}\|_{L^d (B_{2r})} + \|\mathbf{V}\|_{L^\gamma (B_{2r})} \, r \psi^{-1}(r^{-d}) \Big )\,  '\| u_r\|_{W^{1,\psi} (B_{2r})}\\
		& :=\delta(r) \, '\| u_r\|_{W^{1,\psi} (B_{2r})}. 
	\end{align*}

	%%%%%%%%%%%%%%%%%%%%%%%%%%%%%%%%%%%%%%%%%%%%%%%%%%%%%%
	Since $\mathcal{I}_{\psi,1}^{\infty}<\infty$, we have
	\[
	\psi^{-1}(s)=o(s^{1/d}) \qquad \text{as } s\to\infty.
	\]
	Hence, for every $\varepsilon>0$, there exists $s_{\varepsilon}>0$ such that
	\[
	\psi^{-1}(s)\leq \varepsilon s^{1/d}
	\qquad \text{for all } s\geq s_{\varepsilon}.
	\]
	Taking $s=r^{-d}$, we obtain
	\[
	\psi^{-1}(r^{-d})\leq \varepsilon r^{-1}
	\]
	for all sufficiently small $r$. Since $\varepsilon$ is arbitrary, it follows that
	\[
	r\psi^{-1}(r^{-d})\to 0
	\qquad \text{as } r\to 0.
	\]
	
	Moreover
	\[
	\|\eta_r\|_{\infty,B_{2r}}=1,
	\]
	while the continuity of $\mathbf A$ at the origin implies that
	\[
	\|\mathbf{A}-\mathbf A_0\|_{\infty, B_{2r}} \to 0
	\qquad \text{as } r\to 0.
	\]
	Furthermore, the local norms
	\[
	\|\mathbf{B}\|_{L^\psi (B_{2r})},
	\qquad
	\|\mathbf{C}\|_{L^d(B_{2r})},
	\qquad \text{and} \qquad
	\|\mathbf{V}\|_{L^\gamma (B_{2r})}
	\]
	vanish as $r\to 0$. Hence
	\[
	\delta(r)\le \frac12
	\]
	for all sufficiently small $r$. In particular,
	\[
	\|\mathcal{K}_r\|\le \frac12 .
	\]
\end{proof}

We now state the following lemma, which provides a local interior estimate.

\vspace{0.2cm}

\begin{lemma}\label{Lem: principal}
	Let $\psi$ be an $N$-function satisfying the $\Delta_2 \cap \nabla_2$-condition, and let $R>0$. Assume that the coefficients satisfy the $\mathcal{H}_\psi$-condition on $B_R$. Then there exists $0<r_0 \le R/2$ such that the following holds for every $0<r\le r_0$: if $u$ satisfies \eqref{eq.1} with $\operatorname{supp} u \subset B_r$, $\mathbf{F} \in L^\psi(B_{2r})$, and $g \in L^\gamma(B_{2r})$, then
	\begin{equation}
		'\|u\|_{W^{1,\psi}(B_{2r})}
		\le
		c \Big(
		\|\mathbf{F}\|_{L^\psi (B_{2r})}
		+
		\|g\|_{L^\gamma (B_{2r})}
		\Big),
	\end{equation}
	where the positive constant $c$ is independent of $r$.
\end{lemma}

\begin{proof}
	By Proposition~\ref{Thm:TR bound}, one has $\|\mathcal{K}_r\| \le \tfrac12$ for 
	$0 < r \le R/2$. Hence $I - \mathcal{K}_r$ is invertible on $W^{1,\psi}(B_{2r})$. 
	Since $\operatorname{supp} u \subset B_r$, the above decomposition yields $u = u_r$ and
	\[
	u_r - \mathcal{K}_r u_r = \mathbf z_r,
	\]
	where
	\[
	\mathbf z_r = \mathcal{Q}_{2r}[\mathbf{F}] - \mathcal{P}_{2r}[g].
	\]
	By Proposition~\ref{prop:operators},
	\[
	'\|\mathbf z_r\|_{W^{1,\psi}(B_{2r})}
	\lesssim
	\|\mathbf{F}\|_{L^\psi(B_{2r})}
	+
	\|g\|_{L^\gamma(B_{2r})}.
	\]
	Therefore
	\[
	'\|u_r\|_{W^{1,\psi}(B_{2r})}
	\le
	2\, \,  '\|\mathbf z_r\|_{W^{1,\psi}(B_{2r})}
	\lesssim
	\|\mathbf{F}\|_{L^\psi(B_{2r})}
	+
	\|g\|_{L^\gamma(B_{2r})}.
	\]
\end{proof}

\vspace{0.2cm}

We now prove the main result. The proof is based on a standard covering and flattening argument. For each $x_0 \in \overline{\Omega}$, one can find a neighborhood of $x_0$ together with a suitable change of variables mapping it onto either $B_{2r}$ or $B_{2r}^{+} \cup T_{2r}$. The radius $r$ is chosen sufficiently small so that Lemma~\ref{Lem: principal} applies.

\begin{proof}[Proof of Theorem~{\upshape\ref{main theorem}}]

	Assume, by contradiction, that no such constant exists. Then there exist sequences $(\widetilde u_n)_n$, $(\widetilde{\mathbf{F}}_n)_n$, and $(\widetilde g_n)_n$ such that
	\[
	\|\widetilde u_n\|_{W^{1,\psi}(\Omega)}=1,
	\]
	and
	\[
	\|\widetilde u_n\|_{L^1 (\Omega)}
	+\|\widetilde{\mathbf{F}}_n\|_{L^\psi(\Omega)}
	+\|\widetilde g_n\|_{L^\gamma(\Omega)}
	\to 0 .
	\]
	In particular, the sequence $(\widetilde u_n)_n$ is bounded in $W^{1,\psi}(\Omega)$. Hence there exist a subsequence, still denoted by $(\widetilde u_n)_n$, and a function $\widetilde u \in W^{1,\psi}(\Omega)$ such that
	\[
	\widetilde u_n \rightharpoonup \widetilde u
	\quad \text{weakly in } W^{1,\psi}(\Omega).
	\]
	Moreover
	\[
	\widetilde u_n \to \widetilde u
	\quad \text{strongly in } L^\psi(\Omega),
	\]
	and therefore also strongly in $L^1(\Omega)$. Since
	\[
	\|\widetilde u_n\|_{L^1 (\Omega)} \to 0,
	\]
	it follows that $\widetilde u=0$ almost everywhere in $\Omega$.
	
	Fix $k \in \mathbb{N}$ and choose a cut-off function
	\[
	\chi_k \in C_c^\infty(B_{2r}(x_0))
	\]
	such that
	\[
	0\le \chi_k\le 1
	\quad \text{and} \quad
	\chi_k=1 \text{ on } B_{r/2}(x_0).
	\]
	Define
	\[
	\widetilde u_{k,n}(y):=(\chi_k\widetilde u_n)(x_0+y).
	\]
	For simplicity, we use the same notation for the transformed cut-off function and the transformed functions. Then $\widetilde u_{k,n}$ satisfies
	\begin{equation}\label{eq.5}
		\begin{aligned}
			&\sum_{i=1}^{d}\int_{B_{2r}} \partial_{i}v\Bigl(
			\sum_{j=1}^{d}\mathbf{A}_r^{ij} \partial_j \widetilde u_{k,n} 
			+\mathbf{B}_r^i \widetilde u_{k,n}
			+\mathbf{E}_{k,n,r}^i
			\Bigr)\,dy \\
			&\qquad\qquad\qquad\qquad\qquad
			+\int_{B_{2r}} v\Bigl(
			\sum_{i=1}^{d}\mathbf{C}_r^i \partial_i \widetilde u_{k,n}
			+\mathbf{V}_r \widetilde u_{k,n}
			+\mathbf{q}_{k,n,r}
			\Bigr)\,dy
			=0.
		\end{aligned}
	\end{equation}
	where
	\[
	\mathbf{E}_{k,n,r}^i
	:= \chi_k \widetilde{\mathbf{F}}_n^i
	- \sum_{j=1}^d \mathbf{A}_r^{ij}\,\partial_j \chi_{k}\,\widetilde u_{k,n},
	\]
	and
	\[
	\mathbf{q}_{k,n,r}
	:= \chi_k \widetilde g_n
	-\sum_{i=1}^d \mathbf C_r^i \partial_i \chi_{k}\,\widetilde u_{k,n}
	+\sum_{i=1}^d \partial_i \chi_{k}
	\Bigl(
	\sum_{j=1}^d \mathbf{A}_r^{ij}\,\partial_j \widetilde u_{k,n}
	+\mathbf{B}_r^i\,\widetilde u_{k,n}
	+\widetilde{\mathbf{F}}_n^i
	\Bigr).
	\]
	Since $0\le \chi_k\le 1$, it follows that
	\[
	\|\chi_k \widetilde{\mathbf{F}}_n^{i}\|_{L^\psi (B_{2r})}
	\le
	\|\widetilde{\mathbf{F}}_n^{i}\|_{L^\psi (B_{2r})},
	\]
	and the right-hand side tends to $0$. Moreover $\mathbf{A}_r^{ij}$ and $\partial_j \chi_{k}$ are bounded. Hence
	\[
	\|\mathbf{A}_r^{ij}\partial_j \chi_{k}\widetilde u_{k,n}\|_{L^\psi (B_{2r})}
	\lesssim
	\|\widetilde u_{k,n}\|_{L^\psi (B_{2r})}
	\to 0.
	\]
	Consequently $\mathbf{E}_{k,n,r}^{i}$ converges to $0$ in $L^\psi(B_{2r})$.
	
	We decompose $\mathbf{q}_{k,n,r}$ as
	\[
	\mathbf{q}_{k,n,r}
	=
	\mathbf{q}^{(1)}_{k,n,r}
	+
	\mathbf{q}^{(2)}_{k,n,r},
	\]
	where
	\[
	\mathbf{q}^{(1)}_{k,n,r}
	=
	\chi_k \widetilde g_n
	-
	\sum_{i=1}^d \mathbf{C}_r^i \partial_i \chi_{k}\widetilde u_{k,n}
	+
	\sum_{i=1}^d \partial_i \chi_{k}
	\bigl(
	\mathbf{B}_r^i \widetilde u_{k,n}
	+
	\widetilde{\mathbf{F}}_n^i
	\bigr),
	\]
	and
	\[
	\mathbf{q}^{(2)}_{k,n,r}
	=
	\sum_{i,j=1}^d
	\partial_i \chi_{k}\,
	\mathbf{A}_r^{ij}\,
	\partial_j \widetilde u_{k,n}.
	\]
	%%%%%%%%%%%%%%%%%%%%%%%%%%%%%%%%%%%%%%%%%%%%%%%%
	First, we have
	\[
	\|\chi_k\widetilde g_n\|_{L^\gamma (B_{2r})}
	\le \|\widetilde g_n\|_{L^\gamma (B_{2r})}.
	\]
	Since $\gamma\prec\psi$ and $\partial_i \chi_{k}$ is bounded, it also follows that
	\[
	\|\partial_i \chi_{k}\widetilde{\mathbf{F}}_n^i\|_{L^\gamma (B_{2r})}
	\lesssim \|\widetilde{\mathbf{F}}_n^i\|_{L^\gamma (B_{2r})}
	\lesssim \|\widetilde{\mathbf{F}}_n^i\|_{L^\psi (B_{2r})}.
	\]
	Moreover H\"older's inequality yields 
	\[
	\|\mathbf{C}_r^i \partial_i \chi_{k}\widetilde u_{k,n}\|_{L^\gamma (B_{2r})}
	\lesssim \|\mathbf{C}_r^i\|_{L^d (B_{2r})}\,
	\|\widetilde u_{k,n}\|_{L^\psi (B_{2r})},
	\]
	while
	\[
	\|\partial_i \chi_{k}\mathbf{B}_r^i \widetilde u_{k,n}\|_{L^\gamma (B_{2r})}
	\lesssim \|\mathbf{B}_r^i\|_{L^\psi (B_{2r})}
	\|\widetilde u_{k,n}\|_{L^d (B_{2r})}.
	\]
	Since $\mathcal{I}_{\psi,1}^{\infty}<\infty$, one has
	\[
	\psi^{-1}(t)=o(t^{1/d})
	\qquad \text{as } t\to\infty.
	\]
	Therefore
	\[
	\|\widetilde u_{k,n}\|_{L^d (B_{2r})}
	\lesssim \|\widetilde u_{k,n}\|_{L^\psi (B_{2r})}.
	\]
	Consequently
	\[
	\mathbf{q}^{(1)}_{k,n,r}\to 0
	\qquad \text{in } L^\gamma(B_{2r}).
	\]
	Proposition~\ref{prop:operators} then implies that
	\[
	\mathcal{P}_{2r}(\mathbf{q}^{(1)}_{k,n,r})
	\to 0
	\qquad \text{in } L^\psi(B_{2r}).
	\]
	Next, write
	\[
	\mathbf{q}^{(2)}_{k,n,r}
	= \sum_{j=1}^d m_k^j \partial_j\widetilde u_{k,n},
	\]
	where
	\[
	m_k^j
	=\sum_{i=1}^d\partial_i \chi_{k}\mathbf{A}_r^{ij}
	\]
	is bounded on $B_{2r}$. Let $\zeta\in L^{\psi^*}(B_{2r})$ and define
	\[
	\zeta^j=m_k^j\zeta.
	\]
	Then $\zeta^j\in L^{\psi^*}(B_{2r})$, and
	\[
	\int_{B_{2r}} \mathbf{q}^{(2)}_{k,n,r}\zeta\,dy
	= \sum_{j=1}^d \int_{B_{2r}}
	\partial_j\widetilde u_{k,n}\zeta^j\,dy
	\to 0.
	\]
	Thus
	\[
	\mathbf{q}^{(2)}_{k,n,r}
	\rightharpoonup 0
	\qquad \text{in } L^\psi(B_{2r}).
	\]
	The sequence
	\[
	\bigl(\mathcal{P}_{2r}(\mathbf{q}^{(2)}_{k,n,r})\bigr)_n
	\]
	is bounded in $W^{2,\psi}(B_{2r})$ and is therefore relatively compact in $W^{1,\psi}(B_{2r})$. Since its weak limit is $0$, it follows that
	\[
	\mathcal{P}_{2r}(\mathbf{q}^{(2)}_{k,n,r})
	\to 0
	\qquad \text{in } W^{1,\psi}(B_{2r}).
	\]
	Set
	\[
	\mathcal A_{k,n}=\mathcal{P}_{2r}(\mathbf{q}_{k,n,r}).
	\]
	Then
	\begin{equation}\label{eq:6a}
		\sum_{i=1}^d \int_{B_{2r}}
		\partial_{i}v \,\partial_i \mathcal A_{k,n}\,dy
		=
		\int_{B_{2r}} v\,\mathbf{q}_{k,n,r}\,dy,
		\qquad
		\text{for all } v\in C_c^\infty(B_{2r}).
	\end{equation}
	%%%%%%%%%%%%%%%%%%%%%%%%%%%%%%%%%%%%%%%%%%%%%%%%%%%%%%%%
	Define $w_{k,n}=\widetilde u_{k,n}-\mathcal A_{k,n}$. Subtracting \ref{eq:6a} from \ref{eq.5}, we obtain
	\begin{align*}
		& \sum_{i=1}^d  \int_{B_{2r}}
		\partial_i v\left( \sum_{j=1}^d\mathbf{A}_r^{ij}\, \partial_j w_{k,n}
		+\mathbf{B}_r^i w_{k,n}
		+\widetilde {\mathbf{E}}_{k,n,r}^i\right)\,dy \\ 
		&\qquad + \int_{B_{2r}} v\left(\sum_{i=1}^d\mathbf{C}_r^i \, \partial_i w_{k,n}
		+\mathbf{V}_r w_{k,n}
		+\widetilde {\mathbf{q}}_{k,n,r}\right)\,dy=0,
	\end{align*}
	where
	$$
	\widetilde{\mathbf E}_{k,n,r}^i
	=
	\mathbf{E}_{k,n,r}^i
	+\sum_{j=1}^d\mathbf{A}_r^{ij}\, \partial_j\mathcal{A}_{k,n}
	+\mathbf{B}_r^i \mathcal A_{k,n}
	+\partial_i \mathcal A_{k,n},
	$$
	and
	$$
	\widetilde{\mathbf{q}}_{k,n,r}
	=
	\sum_{i=1}^d \mathbf{C}_r^i \, \partial_i\mathcal A_{k,n}
	+\mathbf{V}_r \mathcal A_{k,n}.
	$$
	Since $\mathcal A_{k,n}\to 0$ in $W^{1,\psi}(B_{2r})$ and
	$\mathbf{E}_{k,n,r}^i\to 0$ in $L^\psi(B_{2r})$, it follows that
	$\widetilde {\mathbf{E}}_{k,n,r}^i\to 0$ in $L^\psi(B_{2r})$ and
	$\widetilde{\mathbf{q}}_{k,n,r}\to 0$ in $L^\gamma(B_{2r})$.
	Lemma~\ref{Lem: principal} therefore applies to $w_{k,n}$ for all sufficiently small $r$. Hence
	$$
	'\|w_{k,n}\|_{W^{1,\psi} (B_{2r})}
	\le
	c_4\Bigl(
	\|\widetilde {\mathbf{E}}_{k,n,r}\|_{L^\psi (B_{2r})}
	+
	\|\widetilde {\mathbf{q}}_{k,n,r}\|_{L^\gamma (B_{2r})}
	\Bigr)
	\to 0.
	$$
	Therefore $\widetilde u_{k,n}\to 0$ in $W^{1,\psi}(B_{2r})$. Since
	$(\chi_k)_{k=1}^d$ is a partition of unity on $B_{2r}$, we have
	$$
	\widetilde u_n=\sum_{k=1}^d \widetilde u_{k,n}.
	$$
	Consequently
	$$
	\|\widetilde u_n\|_{W^{1,\psi} (B_{2r})}
	\le
	\sum_{k=1}^d
	\|\widetilde u_{k,n}\|_{W^{1,\psi} (B_{2r})}
	\to 0.
	$$
	This contradicts the fact that
	$$
	\|\widetilde u_n\|_{W^{1,\psi}(\Omega)}=1.
	$$

\end{proof}

In the boundary case, the local problem is posed on \(B_{2r}^{+}\) with homogeneous Dirichlet boundary condition on \(T_{2r}\). For a function \(g\) defined on \(B_{2r}^{+}\), let \(\mathcal{R}(x',x_d)=(x',-x_d)\) and set \(B_{2r}^{-}:=\mathcal{R}(B_{2r}^{+})\). Define the odd extension of \(g\) by
\[
\widetilde g(x)=
\begin{cases}
	g(x), & x_d>0,\\
	-g(\mathcal{R}x), & x_d<0.
\end{cases}
\]
We then define
\[
\mathcal{P}_{2r}(g)
:=
\bigl(\mathcal{P}_{2r}\widetilde g\bigr)\big|_{B_{2r}^{+}}.
\]
By invariance under the reflection \(\mathcal{R}\) and uniqueness of the Dirichlet problem in \(B_{2r}\), the potential \(\mathcal{P}_{2r}(\widetilde g)\) is odd with respect to the hyperplane \(x_d=0\). Consequently
\[
\mathcal{P}_{2r}(g)=0
\qquad \text{on } T_{2r}.
\]
Given a vector field \(\mathbf{F}\) on \(B_{2r}^{+}\), let \(\widetilde{\mathbf{F}}\) denote its even extension across \(T_{2r}\); namely
\[
\widetilde{\mathbf{F}}(x)=
\begin{cases}
	\mathbf{F}(x), & x_d\ge 0,\\
	\mathbf{F}(\mathcal{R}x), & x_d\le 0.
\end{cases}
\]
We define
\begin{align*}
	\mathcal{Q}_{2r}[\mathbf{F}](x)
	=&
	- \sum_{i=1}^d
	\int_{B_{2r}}
	\partial_i \mathbf{k}_{0}(x-\xi)\,
	\widetilde{\mathbf{F}}^{\,i}(\xi)\,d\xi
	\\
	&\quad
	+
	2\sum_{i=1}^{d-1}
	\int_{B_{2r}^{-}}
	\partial_i \mathbf{k}_{0}(x-\xi)\,
	\widetilde{\mathbf{F}}^{\,i}(\xi)\,d\xi
	\\
	=&
	\sum_{i=1}^{d-1}
	\int_{B_{2r}^{-}}
	\Bigl[
	\partial_i \mathbf{k}_{0}(x-\xi)
	-
	\partial_i \mathbf{k}_{0}(x-\mathcal{R}\xi)
	\Bigr]
	\widetilde{\mathbf{F}}^{\,i}(\xi)\,d\xi
	\\
	&\quad
	-
	\int_{B_{2r}^{-}}
	\Bigl[
	\partial_d \mathbf{k}_{0}(x-\mathcal{R}\xi)
	+
	\partial_d \mathbf{k}_{0}(x-\xi)
	\Bigr]
	\widetilde{\mathbf{F}}^{\,d}(\xi)\,d\xi .
\end{align*}
For \(i=1,\ldots,d-1\), one has
\[
\partial_i \mathbf{k}_{0}(x-\mathcal{R}\xi)
=
\partial_i \mathbf{k}_{0}(x-\xi)
\qquad \text{whenever } x_d=0,
\]
whereas
\[
\partial_d \mathbf{k}_{0}(x-\mathcal{R}\xi)
=
-\partial_d \mathbf{k}_{0}(x-\xi).
\]
Hence
\[
\mathcal{Q}_{2r}[\mathbf{F}]=0
\qquad \text{on } T_{2r}.
\]
Therefore Propositions \ref{prop:operators}, \ref{Thm:TR bound}, and Lemma \ref{Lem: principal} remain valid with \(B_{2r}\) replaced by \(B_{2r}^{+}\). This completes the proof of Theorem \ref{main theorem}.

%%%%%%%%%%%%%%%%%%%%%%%%%%%%%%%%%%%%%%%%%%%%%%%%%

\section{Examples}\label{sec5}

Applications of Theorem \ref{main theorem} to standard choices of $\psi$ are presented in this section. We consider only the intermediate regime between $L^d(\Omega)$ and $L^{d+\varepsilon}(\Omega)$, for some $\varepsilon>0$. The relevant computations are summarized in Table \ref{tab:psi}.

\subsection{Lebesgue case}Let $\psi(t)=t^{d+\varepsilon}$ for some $\varepsilon>0$. This corresponds to the supercritical Lebesgue case. Then
	\[
	\psi^*(t)\sim t^{(d+\varepsilon)/(d+\varepsilon-1)}
	\qquad\text{and}\qquad
	\gamma(t)\sim t^{d(d+\varepsilon)/(2d+\varepsilon)}.
	\]
	Define
	\[
	X_{d+\varepsilon}:=L^{d+\varepsilon},
	\qquad
	Y_{d+\varepsilon}:=L^{\frac{d(d+\varepsilon)}{2d+\varepsilon}}.
	\]
	Then
	\[
	\mathcal H_{\psi}:=\left\{
	\begin{array}{l}
		\text{1. $\mathbf A$ is continuous on $\mathcal{O}$,} \\[2pt]
		\text{2. $\mathbf{B}\in X_{d+\varepsilon}(\mathcal{O})$, $\mathbf{C}\in L^d(\mathcal{O})$, and} \\[2pt]
		\text{\phantom{2. }$\mathbf{V}\in Y_{d+\varepsilon}(\mathcal{O})$}
	\end{array}\right\}.
	\]
	Hence the assumptions of Theorem \ref{main theorem} are satisfied. If
	\[
	u\in W^{1}X_{d+\varepsilon}(\Omega)
	\]
	is a solution to \eqref{eq.1}, then
	\[
	\|u\|_{W^{1} X_{d+\varepsilon}(\Omega)}
	\le
	c\Bigl(
	\|u\|_{L^1(\Omega)}
	+
	\|\mathbf{F}\|_{X_{d+\varepsilon}(\Omega)}
	+ 
	\|g\|_{Y_{d+\varepsilon}(\Omega)}
	\Bigr),
	\]
	where $c$ is independent of $u$, $\mathbf{F}$, and $g$. This recovers the classical estimate; see \cite[Theorem 5.5.5', p.~156]{morrey1966multiple}.
\subsection{Orlicz--Zygmund case}
	Two regimes are considered.\\
	\noindent
	(i)~Supercritical Zygmund case: Assume that
		\[
		\psi(t)\sim_{\infty} t^{d+\varepsilon}(\log t)^\alpha,
		\]
		where \(\varepsilon>0\) and \(\alpha\in\mathbb R\). The Young conjugate satisfies
		\[
		\psi^*(t)\sim_{\infty} t^p(\log t)^\beta,
		\]
		where
		\[
		p=\frac{d+\varepsilon}{d+\varepsilon-1}
		\quad\text{and}\quad
		\beta=-\frac{\alpha}{d+\varepsilon-1}.
		\]
		Hence
		\[
		\gamma(t)\sim_{\infty}
		t^{\frac{d(d+\varepsilon)}{2d+\varepsilon}}
		(\log t)^{\frac{\alpha d}{2d+\varepsilon}}.
		\]
		Define
		\[
		X^{\log}_{d+\varepsilon,\alpha}
		:=
		L^{d+\varepsilon}(\log L)^{\alpha},
		\qquad
		Y^{\log}_{d+\varepsilon,\alpha}
		:=
		L^{\frac{d(d+\varepsilon)}{2d+\varepsilon}}
		(\log L)^{\frac{\alpha d}{2d+\varepsilon}}.
		\]
		Then
		\[
		\mathcal H_{\psi}
		:=
		\left\{
		\begin{array}{l}
			\text{1. \(\mathbf{A}\) is continuous on \(\mathcal{O}\),} \\[2pt]
			\text{2. \(\mathbf{B}\in X^{\log}_{d+\varepsilon,\alpha}(\mathcal{O})\), }\\[2pt]
			\text{\phantom{2. } \(\mathbf{C}\in L^d(\mathcal{O})\), and }\\[2pt]
			\text{\phantom{2. } \(\mathbf{V}\in Y^{\log}_{d+\varepsilon,\alpha}(\mathcal{O})\).}
		\end{array}
		\right\}.
		\]
		In this case, if
		\[
		u\in W^{1}X^{\log}_{d+\varepsilon,\alpha}(\Omega)
		\]
		is a solution to \eqref{eq.1}, then
		\[
		\|u\|_{W^{1}X^{\log}_{d+\varepsilon,\alpha}(\Omega)}
		\le
		c\left(
		\|u\|_{L^1(\Omega)}
		+
		\|\mathbf{F}\|_{X^{\log}_{d+\varepsilon,\alpha}(\Omega)}
		+
		\|g\|_{Y^{\log}_{d+\varepsilon,\alpha}(\Omega)}
		\right).
		\]
		(ii)~Critical Zygmund case: Assume that
		\[
		\psi(t)\sim_{\infty}
		t^d(\log t)^{(d-1)+\varepsilon},
		\]
		where \(\varepsilon>0\). The same computation yields
		\[
		\gamma(t)\sim_{\infty}
		t^{d/2}(\log t)^{\frac{d-1+\varepsilon}{2}}.
		\]
		Define
		\[
		X^{\log}_{d,(d-1)+\varepsilon}
		:=
		L^d(\log L)^{(d-1)+\varepsilon},
		\qquad
		Y^{\log}_{d,(d-1)+\varepsilon}
		:=
		L^{\frac d2}(\log L)^{\frac{d-1+\varepsilon}{2}}.
		\]
		Then
		\[
		\mathcal H_{\psi}
		:=
		\left\{
		\begin{array}{l}
			\text{1. \(\mathbf{A}\) is continuous on \(\mathcal{O}\),} \\[2pt]
			\text{2. \(\mathbf{B}\in X^{\log}_{d,(d-1)+\varepsilon}(\mathcal{O})\), }\\[2pt]
			\text{\phantom{2. } \(\mathbf{C}\in L^d(\mathcal{O})\), and }\\[2pt]
			\text{\phantom{2. } \(\mathbf{V}\in Y^{\log}_{d,(d-1)+\varepsilon}(\mathcal{O})\).}
		\end{array}
		\right\}.
		\]
		In this case, if
		\[
		u\in W^{1}X^{\log}_{d,(d-1)+\varepsilon}(\Omega)
		\]
		is a solution to \eqref{eq.1}, then
		\[
		\|u\|_{W^{1}X^{\log}_{d,(d-1)+\varepsilon}(\Omega)}
		\le
		c\left(
		\|u\|_{L^1(\Omega)}
		+
		\|\mathbf{F}\|_{X^{\log}_{d,(d-1)+\varepsilon}(\Omega)}
		+
		\|g\|_{Y^{\log}_{d,(d-1)+\varepsilon}(\Omega)}
		\right).
		\]
	Here, \(c\) is independent of \(u\), \(\mathbf{F}\), and \(g\).

%%%%%%%%%%%%%%%%%%%%%%%%%%%%%%%%%%%%%

\subsection{Double-logarithmic Orlicz space case}
	Let 
	\[
	\psi(t)\sim_{\infty} t^{d+\varepsilon}(\log\log t)^\alpha,
	\]
	where $\varepsilon>0$ and $\alpha\in\mathbb R$. Then
	\[
	\psi^*(t)\sim_{\infty} t^{\frac{d+\varepsilon}{d+\varepsilon-1}}
	(\log\log t)^{-\frac{\alpha}{d+\varepsilon-1}}.
	\]
	It follows that
	\[
	\gamma(t)\sim_{\infty} 
	t^{\frac{d(d+\varepsilon)}{2d+\varepsilon}}
	(\log\log t)^{\frac{\alpha d}{2d+\varepsilon}}.
	\]
	Define
	\[
	X^{\log\log}_{d+\varepsilon,\alpha}
	:=L^{d+\varepsilon}(\log\log L)^\alpha,
	\qquad
	Y^{\log\log}_{d+\varepsilon,\alpha}
	:=L^{\frac{d(d+\varepsilon)}{2d+\varepsilon}}
	(\log\log L)^{\frac{\alpha d}{2d+\varepsilon}}.
	\]
	Then
	\[
	\mathcal H_{\psi}
	:=
	\left\{
	\begin{array}{l}
		\text{1. \(\mathbf{A}\) is continuous on \(\mathcal{O}\),} \\[2pt]
		\text{2. \(\mathbf{B}\in X^{\log\log}_{d+\varepsilon,\alpha}(\mathcal{O})\), }\\[2pt]
		\text{\phantom{2. } \(\mathbf{C}\in L^d(\mathcal{O})\), and }\\[2pt]
		\text{\phantom{2. } \(\mathbf{V}\in Y^{\log\log}_{d+\varepsilon,\alpha}(\mathcal{O})\).}
	\end{array}
	\right\}.
	\]
	If 
	\[
	u\in W^1X^{\log\log}_{d+\varepsilon,\alpha}(\Omega)
	\]
	is a solution to \(\eqref{eq.1}\), then
	\[
	\|u\|_{W^1X^{\log\log}_{d+\varepsilon,\alpha}(\Omega)}
	\le
	c\left(
	\|u\|_{L^1(\Omega)}
	+\|\mathbf{F}\|_{X^{\log\log}_{d+\varepsilon,\alpha}(\Omega)}
	+\|g\|_{Y^{\log\log}_{d+\varepsilon,\alpha}(\Omega)}
	\right),
	\]
	where \(c\) is independent of \(u\), \(\mathbf{F}\), and \(g\).

%%%%%%%%%%%%%%%%%%%%%%%%%%%%%%%%%%%%%%%%%%%%%%%%

\section{Discussions and open problems}\label{sec6}

\subsection{\textbf{Rearrangement-invariant spaces and Schauder-type estimates}}
A natural question is to identify a class \((\mathcal{C})\) of rearrangement-invariant Banach function spaces \(X\) (see \cite{bennett1988interpolation} for background material) such that estimate~\eqref{est.1} holds, and to characterize the corresponding optimal space \(Y\) associated with \(X\).

Let \((X(\Omega),\|\cdot\|_{X(\Omega)})\) be a separable, reflexive rearrangement-invariant Banach function space, and let \(X(0,|\Omega|)\) denote its representation space. If \(f\) is a measurable function on \(\Omega\), and \(f^*\) denotes the non-increasing rearrangement of \(f\) on \((0,|\Omega|)\), then
\[
\|f\|_{X(\Omega)}= \| f^*\|_{X(0,|\Omega|)}.
\]
  Let \(X'(0,|\Omega|)\) denote the associate space of \(X(0,|\Omega|)\), equipped with the norm
  \[
  \|g\|_{X'(0,|\Omega|)}
  =
  \sup \left\{
  \int_0^{|\Omega|} |f(s)g(s)|\,ds
  :\,
  f\in X(0,|\Omega|),
  \ \|f\|_{X(0,|\Omega|)}\le 1
  \right\}.
  \]
The product space is  denoted by  
\[
L^d\odot X
:=
\bigl\{
h \, \, \text{measurable}:\ h=fg \text{ a.e. in }\Omega,\ f\in L^d(\Omega),\ g\in X(\Omega)
\bigr\},
\]
endowed with the norm
\[
\|h\|_{L^d\odot X}
:=
\inf\bigl\{
\|f\|_{L^d(\Omega)}\,\|g\|_{X(\Omega)}:\ h=fg \text{ a.e. in }\Omega
\bigr\}.
\]
Let \(\varphi_X\) denote the fundamental function of \(X\), defined by \(\varphi_{X}(t)=\|\chi_E\|_{X}\) whenever \(|E|=t\). Assume that $t^{1/d}/\varphi_X(t)\to 0$ as
$t\to 0^+$. Denote by \(\underline{\beta}_X\) and \(\overline{\beta}_X\) the lower and upper Boyd indices of \(X\), respectively. 
For \(a>1\), define the weighted Stieltjes transform \(S_a\) by
\[
(\mathrm{S}_af)(t)=t^{\frac1a-1}\int_0^t f(s)\,ds+\int_t^{|\Omega|}f(s)s^{\frac1a-1}\,ds,\qquad t\in(0,|\Omega|).
\]
Let 
\begin{equation}\label{classe-C}
	(\mathcal{C}) : \left\{
	\begin{aligned}
		(\mathcal{C}-i)\quad
		& 0<\underline{\beta}_X\le \overline{\beta}_X<1,
		\\
		(\mathcal{C}-ii)\quad
		& \text{assume that } (1+t)^{\frac d\alpha-1}\in X(0,|\Omega|), \\
		&Y=\left\{f:\ \|\mathrm{S}_{d/\alpha}(f^*)\|_{X(0,|\Omega|)}<\infty\right\}.
		\\  
		(\mathcal{C}-iii)\quad
		&  L^d \odot X \hookrightarrow Y 
		\\
		(\mathcal{C}-iv)\quad
		& \varpi_X(r):=\|s^{-1+1/d}\chi_{(0,r^d)}(s)\|_{X'(0,|\Omega|)} \to 0
		\quad \text{as } r\to 0^+,
		%\|\cdot\|_{\infty,B_r}\lesssim\frac{r}{\varphi_X(|B_{2r}|)}\,\|\cdot\|_{W^1X(B_r)}'.
	\end{aligned}
	\right.
\end{equation}
Condition \((\mathcal {C}-ii)\) guarantees the existence of an optimal space \(Y\), associated with \(X\), such that the Riesz potential operator is bounded from \(Y\) into \(X\); see \cite[\textup{Theorem}~6.4]{edmunds1995double}. As a consequence of our argument, we obtain the following result.

\begin{theo}
	Let \(X\in (\mathcal{C})\), and assume that \(\mathbf F\in X(\Omega)\) and \(g\in Y(\Omega)\). Then every weak solution \(u\in W^{1}X(\Omega)\) to equation~\eqref{eq.1} satisfies estimate~\eqref{est.1}; namely,
	\[
	\|u\|_{W^{1,X}(\Omega)}
	\le
	c\left(
	\|u\|_{L^1(\Omega)}
	+
	\|\mathbf F\|_{X(\Omega)}
	+
	\|g\|_{Y(\Omega)}
	\right),
	\]
	where \(c=c(d,\Omega,m,M,\omega_{\mathbf A},X,\Lambda)\) is a positive constant independent of \(u\), \(\mathbf F\), and \(g\). Here,
	\[
	\|\mathbf B\|_{X(\Omega)}
	+
	\|\mathbf C\|_{L^d(\Omega)}
	+
	\|\mathbf V\|_{Y(\Omega)}
	\le \Lambda .
	\]
\end{theo}

 In the Orlicz setting \(X=L^\psi\), a sharp characterization of Condition \((\mathcal {C}-ii)\) was obtained in \cite{cianchi1999strong}. In this case, the optimal associated space is \(Y=L^\gamma\), and the criterion is expressed in terms of the Riesz potential
\[
I_{\alpha}f(x)=\int_{\mathbb{R}^{d}}\frac{f(y)}{|x-y|^{\,d-\alpha}}\,dy,
\qquad 0<\alpha<d.
\]
To this end, one defines the Young functions
\[
\gamma_{d/\alpha}(s)=\int_{0}^{s} r^{\alpha/(d-\alpha)} 
\left(\Phi_{d/\alpha}^{-1}\left(r^{d/(d-\alpha)}\right)\right)^{d/(d-\alpha)}\,dr,
\qquad
\Phi_{d/\alpha}(s)=\int_{0}^{s}\frac{\gamma^*(t)}{t^{1+d(d-\alpha)}}\,dt,
\]
and 
\[
\psi_{d/\alpha}^*(s)=\int_{0}^{s} r^{\alpha/(d-\alpha)}
\left(\Psi_{d/\alpha}^{-1}\left(r^{d/(d-\alpha)}\right)\right)^{d/(d-\alpha)}\,dr,
\qquad
\Psi_{d/\alpha}(s)=\int_{0}^{s}\frac{\psi(t)}{t^{1+d/(d-\alpha)}}\,dt.
\]
Then $I_\alpha$ is bounded from $L^\gamma(\mathbb{R}^d)$ to $L^\psi(\mathbb{R}^d)$ if and only if
\[
\int_{0}\frac{\psi(t)}{t^{1+d/(d-\alpha)}}\,dt<\infty,
\qquad
\int_{0}\frac{\gamma^*(t)}{t^{1+d/(d-\alpha)}}\,dt<\infty,
\]
$\psi_{d/\alpha}\prec \gamma$ and $\psi\prec \gamma_{d/\alpha}$; see \cite[Theorem~2($ii$)]{cianchi1999strong}. 

The classical Lebesgue framework is recovered by choosing
\[
\psi(s)=s^{\frac{dp}{d-\alpha p}},
\qquad 1<\alpha p<d,
\]
which yields the associated space
\[
\gamma(s)=s^{p}.
\]
In the borderline case, \cite{o1960fractional} proved that the choice
\[
\psi(s)=s^{\frac{d}{d-\alpha}}
\]
admits the optimal domain space
\[
\gamma(s)=s\log^{1-\alpha/d}(1+s)
\]
on sets of finite measure. At the exponential critical threshold, \cite{strichartz1972note,trudinger1967imbeddings} established that
\[
\psi(s)=\exp\!\left(s^{\frac{d}{d-\alpha}}\right)-1
\]
is optimally associated with
\[
\gamma(s)=s^{d/\alpha}.
\]
Further refinements were obtained in \cite{edmunds1995double}, where it was shown, in particular, that the double-exponential space
\[
\psi(s)=\exp\!\left(\exp\!\left(s^{\frac{d}{d-\alpha}}\right)\right)-e
\]
corresponds to
\[
\gamma(s)=s^{d/\alpha}\log^{(d-\alpha)/d}(1+s).
\]

Let \(X=L^{p,q}(\Omega)\) be a Lorentz space over a domain \(\Omega\subset \mathbb R^d\), with \(1<p<\infty\) and \(1\le q\le \infty\). Condition \((\mathcal {C}-i)\) is fulfilled in the whole range of parameters; see \cite{hunt1966spaces,sharpley1988interpolation}. Concerning condition \((\mathcal {C}-ii)\), the optimal associated space is given by \(Y=L^{r,q}(\Omega)\), where \(r\) is defined through
\[
\frac1p=\frac1r-\frac1d,
\]
according to the boundedness properties of the Riesz potential of order \(1\); see \cite{o1963convolution}.

Moreover if \(\Omega\) is a bounded regular domain, then
\[
W^{1}L^{p,q}(\Omega)\hookrightarrow C(\Omega)
\]
whenever \(p>d\), for every \(q\), and also in the limiting case \(p=d\) and \(q=1\); see \cite{cianchi1998sobolev,ranjbar2009embedding}. Finally, condition \((\mathcal {C}-iv)\) follows again from \cite{o1963convolution}.

Hence the Lorentz space \(L^{p,q}(\Omega)\) satisfies all assumptions in \((\mathcal C)\) precisely when either \(p>d\), or \(p=d\) and \(q=1\).

\subsection{\textbf{Open problems}}

Let $\mathbf A, \mathbf B, \mathbf C, \mathbf V$ satisfy the structural assumptions associated with \eqref{eq.1}, and let $u \in W^{1}X(\Omega)$ be a weak solution.

We ask whether the \textit{a priori} estimate
\[
\|u\|_{W^{1}X(\Omega)}
\lesssim
\|u\|_{L^{1}(\Omega)}
+
\|\mathbf F\|_{X(\Omega)}
+
\|g\|_{Y(\Omega)}
\]
is equivalent to the boundedness of the Riesz potential operator
\[
I_1 : Y(\Omega) \to X(\Omega),
\]
in the framework of general rearrangement-invariant Banach function spaces, beyond the classical Orlicz and Lorentz scales.

More precisely, does this equivalence remain valid over the full class of rearrangement-invariant spaces, without additional assumptions on the Boyd indices or on the growth properties of the fundamental function?

\subsection*{Acknowledgements}The first author gratefully acknowledges the support of the National Centre for Scientific and Technical Research (CNRST), Morocco, through the PhD Associate Scholarship Program (PASS).

\bibliographystyle{abbrv}
\bibliography{sn-bibliography}
\end{document}